\documentclass[10pt,reqno]{amsart}
\usepackage{latexsym,amssymb,amsmath}
\usepackage{amsmath,amsfonts}
\usepackage{epsfig}
\usepackage{graphicx}
\usepackage{verbatim,xcolor}
\usepackage{dirtytalk}

\usepackage{dsfont}

\usepackage{cite}

\usepackage{enumitem}
\usepackage{mathabx}
\usepackage{cancel}

\usepackage{xcolor}
\usepackage{scalerel,stackengine}

\usepackage[mathscr]{euscript}
\usepackage{dsfont}
\usepackage{qtree}

\DeclareMathOperator\supp{supp}

\newcommand{\norm}[1]{\left\lVert#1\right\rVert}
\theoremstyle{plain}
\newtheorem{theorem}{Theorem}[section]

\newtheorem{Corollary}[theorem]{Corollary}
\newtheorem{lemma}[theorem]{Lemma}
\newtheorem{Definition}[theorem]{Definition}

\theoremstyle{remark}
\newtheorem{Remark}[theorem]{Remark}

\allowdisplaybreaks

\newcommand{\Hmm}[1]{\leavevmode{\marginpar{\tiny%
$\hbox to 0mm{\hspace*{-0.5mm}$\leftarrow$\hss}%
\vcenter{\vrule depth 0.1mm height 0.1mm width \the\marginparwidth}%
\hbox to 0mm{\hss$\rightarrow$\hspace*{-0.5mm}}$\\\relax\raggedright #1}}}

\usepackage{scalerel,stackengine}
\stackMath
\newcommand\reallywidehat[1]{%
\savestack{\tmpbox}{\stretchto{%
  \scaleto{%
    \scalerel*[\widthof{\ensuremath{#1}}]{\kern-.6pt\bigwedge\kern-.6pt}%
    {\rule[-\textheight/2]{1ex}{\textheight}}
  }{\textheight}%
}{0.5ex}}%
\stackon[1pt]{#1}{\tmpbox}%
}

\numberwithin{equation}{section}
\begin{document}
\author{E. Ba\c{s}ako\u{g}lu}
\author{C. Sun}
\author{N. Tzvetkov}
\author{Y. Wang}
\address{Institute of Mathematical Sciences, ShanghaiTech University, Shanghai, 201210, China}
\email{ebasakoglu@shanghaitech.edu.cn}
\address{CNRS, Université Paris-Est Créteil, Laboratoire d’Analyse et de Mathématiques appliquées,
UMR 8050 du CNRS, 94010 Créteil cedex, France}
\email{chenmin.sun@cnrs.fr}

\address{Ecole Normale Supérieure de Lyon, UMPA, UMR CNRS-ENSL 5669, 46, allée d’Italie, 69364-Lyon Cedex 07, France}
\email{nikolay.tzvetkov@ens-lyon.fr}

\address{School of Mathematics, University of Birmingham, Watson Building, Edgbaston, Birmingham B15 2TT, United Kingdom}
\email{y.wang.14@bham.ac.uk}

\begin{abstract}
     We study the Boltzmann equation with the constant collision kernel in the case of spatially periodic domain $\mathbb{T}^d$, $d\geq 2$. Using the existing techniques from nonlinear dispersive PDEs, we prove the local well-posedness result in $L^{2,r}_vH^s_x$ for $s>\frac{d}{2}-\frac{1}{4}$ and $r>\frac{d}{2}$. To reach the result, the main tool we establish is the $L^4$ Strichartz estimate for solutions to the corresponding linear equation.
\end{abstract}
\title[Well-posedness for Boltzmann Equation]{Local well-posedness for the periodic Boltzmann equation with constant collision kernel}
\maketitle
\section{Introduction}
The Boltzmann equation is the fundamental equation of collisional kinetic theory, which describes the state of the dilute gas (or plasma) modelled by a distribution function $f(t, x, v)\geq 0$ in the particle phase space. In this paper, we are concerned with the periodic Cauchy problem for Boltzmann equation with constant collision kernel, which is written in the following form: 
\begin{equation}\label{BoltzmannIVP}
    \begin{cases}
        \partial_{t}f+v\cdot \nabla_xf=Q(f,f),\\
        f(0,x,v)=f_{0}(x,v),
    \end{cases}
\end{equation}
where $x\in \mathbb{T}^d,\, v\in \mathbb{R}^d$, $d\geq 2$, correspond to the position and accompanying velocity of a typical particle respectively, and the quadratic operator $Q$ (conventionally, $Q$ is called as collision operator) is defined by 
\begin{align*}
 Q(f,f)(t, x, v)=\int_{\mathbb{S}^{d-1}}\int_{\mathbb{R}^d}[f(t,x,u^*)f(t,x,v^*)-f(t,x,u)f(t,x,v)]\,dud\omega,    
\end{align*}
which describes the interaction between colliding particles. From the viewpoint of physical interpretation, the variables $(u,v)$ represent the velocities for a pair of particles before collision, and the variables $(u^*, v^*)$ represent the respective
velocities after collision, which are given by
\begin{align*}
    u^*=u+(\omega\cdot (v-u))\omega,\quad
    v^*=v-(\omega\cdot (v-u))\omega.
\end{align*}
In the above formulation, the unit vector $\omega\in \mathbb{S}^{d-1}\subset \mathbb{R}^d$ is a parameter associated to the deflection angle in the collision of the particles. The Boltzmann collision operator can then be split into a gain and a loss terms
\begin{align*}
    Q(f,g)=Q^+(f,g)-Q^-(f,g)
\end{align*}
where the gain term is \begin{align*}
  Q^+(f,g)=\int_{\mathbb{S}^{d-1}}\int_{\mathbb{R}^d}f(v^*)g(u^*)\,dud\omega 
\end{align*}
and the loss term is
\begin{align*}
Q^-(f,g)=\int_{\mathbb{S}^{d-1}}\int_{\mathbb{R}^d}f(v)g(u)\,dud\omega=c f(v)\int_{\mathbb{R}^d}g(u)\,du.
\end{align*}
Since the collisions in thermodynamical system are assumed to be elastic for the Boltzmann equation, the following conservation relations of momentum and kinetic energy hold: 
\begin{align}\label{conservation}
  u+v=u^*+v^*,\,\,\, |u|^2+|v|^2=|u^*|^2+|v^*|^2. 
\end{align}

Investigating local/global well-posedness for the Boltzmann equation is of great interest among mathematicians, which can also be useful for the study of other problems such as the mathematical derivation of hydrodynamic limits of the Boltzmann equation, and the derivation of the Boltzmann equation from quantum many-body dynamics or the classical particle systems, etc. In this respect, a large number of mathematical studies have been devoted to the Boltzmann equation up until now, in which the various mathematical frameworks have been developed in order to construct solutions in different settings. For instance, we refer the readers to \cite{Alexandreglobal, Alexandre, Ampatzoglou, Arsenio, CDP, CDPmoment, CDPsmall, Chen-Holmer, Chen-Shen-Zhang, Chen-Shen-Zhang2, DiPerna, DuanR, DuanR2, Hel, Hel2, KanielShinbrot, Sohinger, Toscani} when the spatial domain is $\mathbb{R}^n$. In a recent series of papers \cite{CDP, CDPmoment, CDPsmall}, T. Chen, Denlinger and Pavlović introduced a new approach based on the Wigner transform and techniques from nonlinear dispersive PDEs on the study of the quantum many-body hierarchy dynamics in order to establish well-posedness theory of the Boltzmann equation below the continuity threshold. In particular, in \cite{CDP}, the usual regularity requirement $s>\frac{d}{2}$ for well-posedness has been relaxed to $s>\frac{d-1}{2}$ for both Maxwellian molecules and hard potentials with cut-off. Later, X. Chen and Holmer \cite{Chen-Holmer} showed that the regularity $s=1$ is actually the  well/ill-posedness Sobolev regularity threshold for the Boltzmann equation in the 3D constant kernel case. In a subsequent paper \cite{Chen-Shen-Zhang}, the ill-posedness result in \cite{Chen-Holmer} has been extended to the case of 2D/3D kernel with soft potentials by proving that the well/ill-posedness separation point is also $\frac{d-1}{2}$. Besides, the well-posedness proved in \cite{Chen-Shen-Zhang} has completed the 2D/3D soft potential kernel case left from \cite{CDP}. 

As for the spatially periodic case, there have been further developments of the Boltzmann local/global existence theory under certain conditions on the collision kernel since the pioneering works of Grad \cite{GradL} (local existence theory) and Ukai \cite{UkaiS} (global existence theory). Next we shall review some of the results associated to the periodic spatial domain. With the initial data which are perturbations of Maxwellian equilibrium states $\mu(v)=e^{-|v|^2}$, the result in \cite{UkaiS} demonstrated the existence of a unique global solution 
\begin{align}\label{perturofeq}
 f=\mu(v)+\sqrt{\mu(v)}g   
\end{align}
to the Boltzmann equation with cut-off hard potential kernel for which ($g$ solves the reformulated Boltzmann equation based on Maxwellian distribution $\mu(v)$)
\begin{align*}
 g\in L_t^{\infty}([0,\infty);H_{l,\beta})\cap C_t^{\infty}([0,\infty);H_{l,\beta})\,\,\,\text{for}\,\,\,l>\frac{3}{2}, \beta> \frac{5}{2} 
\end{align*}
with the smallness assumption on the initial data $g_0=g|_{t=0}$, and where the Banach space $H_{l,\beta}$ is defined by means of the norm $\Vert g\Vert_{H_{l,\beta}}=\sup_{v\in\mathbb{R}^3}(1+|v|)^{\beta}\Vert g\Vert_{H^l_x(\mathbb{T}^3)}$. Note that such solutions as in \eqref{perturofeq} are regarded as linearization of the Boltzmann equation under consideration around the Maxwellian equilibrium state. We further note that the subsequent works in this direction mentioned in what follows also study solutions which are perturbations of Maxwellian equilibrium states as in \eqref{perturofeq}, whereas in the current paper we consider solutions not given by such linearization. Indeed, we follow the scheme of \cite{Chen-Holmer, Chen-Shen-Zhang} in constructing the solutions within the framework of a periodic Cauchy problem, in which the Boltzmann equation is addressed via harmonic analysis methods such as the Littlewood-Paley and the Fourier restriction space theories. For the periodic Boltzmann equation without angular cut-off, Gressman and Strain \cite{Gressman} first established the global-in-time existence theory by constructing the strong solution $f(t,x,v)$ in the form \eqref{perturofeq} that satisfies, for small initial data $g_0=g|_{t=0}$, $g\in L^{\infty}_tL^2_vH^2_x((0,\infty)\times \mathbb{T}^3_x\times \mathbb{R}^3_v)$ for hard potential case, while $g\in L^{\infty}_tH^4_{x,v}((0,\infty)\times \mathbb{T}^3_x\times \mathbb{R}^3_v)$ for soft potential case. In the setting that the standard perturbation $g$ as in \eqref{perturofeq} to $\mu(v)$ is defined, an energy method for the Boltzmann equation in a periodic box was first developed by Guo \cite{GuoY} to construct global classical solutions under the Grad angular cut-off assumption in the case of soft potentials, see also the work of Liu, Yang and Yu \cite{LiuYang} that introduced an energy method for the time-asymptotic, nonlinear stability of the global Maxwellian states in a spatially unbounded domain. Regarding the Boltzmann equation with cut-off hard potentials in a general bounded domain, Guo \cite{GuoY1} developed a unified $L^2-L^{\infty}$ theory to obtain the global existence of $L^{\infty}_{x,v}$ solutions. For a large amplitude periodic initial data, Duan, Huang, Wang and Yang \cite{DuanR} established the well-posedness theory of the Boltzmann equation by developing $L^{\infty}_xL^1_v\cap L^{\infty}_{x,v}$ method for both hard and soft potentials with angular cut-off. Later, motivated by \cite{GuoY1, DuanR}, the small global existence theory of the angular cut-off Boltzmann equation near a global Maxwellian in the space $L^r_vL^{\infty}_{t,x}$ for $1<r\leq \infty$ with certain weights was established by Nishimura \cite{Nishimura}. Lastly, in the more recent study \cite{DuanSakamoto}, Duan, Liu, Sakamoto and Strain considered the non-cut-off Boltzmann equation in a torus $\mathbb{T}^3$, in which they proved global-in-time existence of small amplitude solutions in the spaces $\ell^1_kL^{\infty}_TL^2_v$ with suitable exponential weights.

In order to make use of dispersive PDE techniques in constructing local solutions of the Boltzmann equation \eqref{BoltzmannIVP}, taking the inverse Fourier transform in the velocity variable $v$ only, and denoting $\tilde{f}(t,x,\xi)=\mathcal{F}_{v\mapsto\xi}^{-1}[f]$, we rewrite \eqref{BoltzmannIVP} as
\begin{equation}\label{hyperbolicschrdingerIVP}
    \begin{cases}
        i\partial_{t}\tilde{f}+\nabla_{\xi}\cdot \nabla_x\tilde{f}=i\mathcal{F}_{v\mapsto\xi}^{-1}[Q(f,f)],\\
        \tilde{f}(0,x,\xi)=\tilde{f}_{0}(x,\xi),
    \end{cases}
    \quad (t,x,\xi)\in[-T,T]\times\mathbb{T}^{d}\times \mathbb{R}^d.
\end{equation}
Let us denote $\tilde{Q}(f,g):=\mathcal{F}_{v\mapsto\xi}^{-1}[Q(f,f)]$. In this setting, by the Bobylev identity \cite{bobylev}, the loss and gain operators turn into
\begin{equation}\label{lossandgainop}
 \begin{aligned}
 &\tilde{Q}^+(\tilde{f},\tilde{g})(\xi):=\mathcal{F}_{v\mapsto\xi}^{-1}[Q^+(f,g)](\xi)= \int_{\mathbb{S}^{d-1}}\tilde{f}(\xi^+) \tilde{g}(\xi^-) \,d\omega\\&\tilde{Q}^-(\tilde{f},\tilde{g})(\xi):=\mathcal{F}_{v\mapsto\xi}^{-1}[Q^-(f,g)](\xi)=\tilde{f}(\xi)\tilde{g}(0)
\end{aligned}   
\end{equation}
where $\xi^+=\frac{1}{2}(\xi+|\xi|\omega)$ and $\xi^-=\frac{1}{2}(\xi-|\xi|\omega)$. The linear hyperbolic Schrödinger equation 
\begin{align*}
 i\partial_{t}\phi+\nabla_{\xi}\cdot \nabla_x\phi=0,\quad \phi|_{t=0}=\phi_{0},  
\end{align*}
for $(x,\xi)\in \mathbb{R}^d\times \mathbb{R}^d$, admits solutions satisfying the Strichartz estimates
\begin{align}\label{Strichartzforeuclidancase}
    \Vert e^{it\nabla_{\xi}\cdot \nabla_x}\phi_0\Vert_{L^q_{t}L^p_{x\xi}}\lesssim \Vert  \phi_0\Vert_{L^2_{x\xi}},\quad\frac{2}{q}+\frac{2d}{p}=d,\,\,q\geq 2,\,\,d\geq2.
\end{align}
The Strichartz estimates \eqref{Strichartzforeuclidancase} essentially follow from Theorem 1.2 in \cite{Keeltao}, see the appendix of \cite{Chen-Shen-Zhang, Chen-Shen-Zhang2} for a discussion. Also, it is is worth noting that the sharpness arguments in \cite{Chen-Holmer, Chen-Shen-Zhang, Chen-Shen-Zhang2} essentially follow from the estimates \eqref{Strichartzforeuclidancase}, which have no loss of derivatives. As there is no similar estimate obtained previously for the spatially periodic case, we will need to derive particular Strichartz estimate that works well in the nonlinear context for the IVP \eqref{hyperbolicschrdingerIVP}, yet with certain loss of derivatives due to the hyperbolic nature of our problem posed in a periodic box, see Section \ref{sectStrichartz}.
\subsection{Notation}
The Fourier transform of $g$ defined on $\mathbb{T}^d\times \mathbb{R}^d$ is given by
\begin{align*}
    \mathcal{F}_{x, \xi}[g](n,v)=\widehat{g}(n,v)=\int_{\mathbb{T}^d\times \mathbb{R}^d}g(x,\xi)e^{-i(n\cdot x+v\cdot\xi)}\,dx\,d\xi.
\end{align*}
The partial Fourier transforms are also defined accordingly. Let $x\in\mathbb{T}^d,\, \xi\in\mathbb{R}^d$ and $N,M$ be dyadic numbers. We denote the frequency localization operator of $x/\xi$ at frequency $N/M$ by $P_N^x/P^{\xi}_M$. For brevity, we sometimes write $P_{N,M}$ (or $P_N/P_M$) to denote $P_N^xP_M^{\xi}$ (or $P_N^x/P^{\xi}_M$) and use both of these notations interchangeably. Let $\varphi$ be a smooth function satisfying $\varphi(x)=1$ for all $|x|\leq 1$ and $\varphi(x)=0$ for $|x|\geq 2$. We then form the function $\varphi_M(v)=\varphi(\frac{v}{M})-\varphi(\frac{2v}{M})$, $v\in \mathbb{R}^d$, and set $P_M^{\xi}g(\xi)=\mathcal{F}^{-1}_v[\varphi_M(v)\widehat{g}(v)](\xi)$. For any set $S\subset \mathbb{R}^d$, we shall define $P_S^x$ (or $P_S$) as the multiplier operator $\widehat{P_S^xg}(n)=\mathds{1}_S\widehat{g}(n)$, where we denote $\mathds{1}_S$ to be the frequency projection on $S\cap \mathbb{Z}^d$. Thus, we define the Littlewood-Paley projector $P_N^x$ as follows
\begin{align*}
  \mathcal{F}_x[{P_N^xg}](n)= \mathcal{F}_x[P^x_{(-N,N]^d\setminus(-\frac{N}{2},\frac{N}{2}]^d}g](n)=\mathds{1}_{(-N,N]^d\setminus(-\frac{N}{2},\frac{N}{2}]^d}\widehat{g}(n),\quad n\in \mathbb{Z}^d. 
\end{align*}
Let $a+(-N,N]^d$ be a square of size $N$ in $\mathbb{Z}^d$ centered at $a\in \mathbb{Z}^d$. Then, with the above notation, the projection operator $P_{a+(-N,N]^d}$ is defined by $\mathcal{F}[P_{a+(-N,N]^d}g](n)=\mathds{1}_{a+(-N,N]^d}\widehat{g}(n)$.
We denote the hyperbolic Schr\"odinger propagator by the following
\begin{align}\label{hyperbolicschrödingergroup}
S(t)g(x,\xi):=e^{it\nabla_x\cdot\nabla_{\xi}}g(x,\xi)=\mathcal{F}^{-1}_{n,v}[e^{-itv\cdot n}\widehat{g}(n,v)](x,\xi).
\end{align}
We reserve the notation $\psi(t)$ for a smooth compactly supported function such that $\supp{\psi}\subset (-2,2)$ and $\psi\equiv 1$ on $[-1,1]$. Also, let $\psi_{T}(t)=\psi(\frac{t}{T})$ for $T>0$. Lastly, we denote $L^p_T:=L^p([-T,T])$ for $1\leq p \leq \infty$.
\subsection{Function Spaces}
We define the Sobolev space $H^{s,r}_{x\xi}(\mathbb{T}^d\times \mathbb{R}^d)=H^{s}_xH^{r}_{\xi}(\mathbb{T}^d\times \mathbb{R}^d)$ by the norm
\begin{align*}
    \Vert \tilde{f}\Vert_{H^{s,r}_{x\xi}}=\Vert \langle \nabla_x\rangle^s\langle \nabla_{\xi}\rangle^r\tilde{f}\Vert_{L^2_{x\xi}}=\Vert \langle \nabla_x\rangle^s\langle v\rangle^r f\Vert_{L^2_{xv}}=\Vert f\Vert_{L^{2,r}_vH^s_x}.
\end{align*}
Our analysis of Boltzmann equation relies very much on the Fourier restriction space theory which is introduced in the context of Schrödinger equation by Bourgain in  \cite{Bourgain1993}. The Fourier restriction space $X^{s,r,b}(\mathbb{R}\times \mathbb{T}^d\times \mathbb{R}^d)$ associated to the semigroup \eqref{hyperbolicschrödingergroup} is defined via the norm
\begin{align}\label{Fourierrestrictionspace}
 \Vert \tilde{f}\Vert_{X^{s,r,b}}=\Vert \langle \tau+n\cdot v\rangle^b\langle n\rangle^s\langle v\rangle^r\mathcal{F}_{t,x,\xi}[\tilde{f}]\Vert_{\ell_n^2L_{\tau,v}^2}=\Vert \langle \tau+n\cdot v\rangle^b\langle n\rangle^s\langle v\rangle^r\mathcal{F}_{t,x}[f]\Vert_{\ell_n^2L_{\tau,v}^2}.
 \end{align}
 For simplicity, we shall write $X^{s,b}:=X^{s,s,b}$. We also use the restricted $X^{s,r,b}$ space, $X^{s,r,b}_T$, which is the equivalence classes of functions that agree on $[-T,T]$ endowed with the norm
 \begin{align*}
  \Vert \tilde{f}\Vert_{X^{s,r,b}_T}=\inf_{F\arrowvert_{[-T,T]}=\tilde{f}}\Vert F\Vert_{X^{s,r,b}}. 
 \end{align*}
 For the notion of conditional uniqueness of a solution $f$ to \eqref{BoltzmannIVP} (which will be introduced in the next subsection), we shall define the space $Y^{s,r,b}$ associated with $X^{s,r,b}$ by the norm
 \begin{align}\label{Y^s,r,b}
     \Vert f\Vert_{Y^{s,r,b}}:=\Vert \langle \tau+n\cdot v\rangle^b\langle n\rangle^s\langle v\rangle^r\mathcal{F}_{t,x}[f]\Vert_{\ell_n^2L_{\tau,v}^2}.
 \end{align}
 We also define $Y^{s,b}$ and the time-restricted space $Y^{s,r,b}_T$ analogously as above. Note  that, in view of \eqref{Fourierrestrictionspace} and \eqref{Y^s,r,b}, we see that a solution $\tilde{f}\in X^{s,r,b}$ to \eqref{hyperbolicschrdingerIVP} corresponds to a solution $f\in Y^{s,r,b}$ to \eqref{BoltzmannIVP}. Note also that $Y^{s,r,b}_T$ embeds into $C\big([-T,T]; L^{2,r}_vH^s_x\big)$ for $b>\frac{1}{2}$. In the light of these, we introduce the notion of well-posedness and state our result for the Boltzmann equation in the following subsection.
 \subsection{Main Result}
\begin{Definition}
    We say that the initial value problem \eqref{BoltzmannIVP} is locally well-posed in $L^{2,r}_vH^s_x$, if for all $R>0$, there exists $T=T(R)$ such that the followings hold.
    \begin{itemize}
        \item[$(a)$] \emph{(Existence and uniqueness)} For each $f_0\in L^{2,r}_vH^s_x$ with $\Vert f_0\Vert_{L^{2,r}_vH^s_x}\leq R$, there exists a unique  solution $f(t,x,v)$ to the integral equation of \eqref{BoltzmannIVP} in $Y^{s,r,\frac{1}{2}+}_T\hookrightarrow C\big([-T,T]; L^{2,r}_vH^s_x\big)$ (equivalently, corresponding solution $\tilde{f}$ to \eqref{hyperbolicschrdingerIVP} is unique in $X^{s,r,\frac{1}{2}+}_T$). Furthermore, we have $f(t,x,v)\geq 0$ if $f_0\geq 0$.
        \item[$(b)$] The data-to-solution map $f_0\mapsto f$ is uniform continuous from $L^{2,r}_vH^s_x$ to $C\big([-T,T];L^{2,r}_vH^s_x\big)$. Indeed, if $f$ and $g$ are two solutions to \eqref{BoltzmannIVP} on $[-T,T]$, then for all $\epsilon>0$, there exists $\delta(\epsilon)$ independent of $f$ or $g$ such that
        \begin{align*}
            \Vert f(t)-g(t)\Vert_{C([-T,T];L^{2,r}_vH^s_x)}<\epsilon\,\,\text{provided that}\,\, \Vert f(0)-g(0)\Vert_{L^{2,r}_vH^s_x}<\delta(\epsilon).
        \end{align*}
    \end{itemize}
 \end{Definition}
 \begin{theorem}\label{maintheorem}
     Let $d\geq 2$. Then, the Cauchy problem \eqref{BoltzmannIVP} is locally well-posed in $L^{2,r}_vH^s_x$ for $s>\frac{d}{2}-\frac{1}{4}$ and $r>\frac{d}{2}$.
 \end{theorem}
\begin{Remark}
Despite taking $\mathbf{b}\equiv1$ in our analysis to establish Theorem \ref{maintheorem}, we may also allow the following collision operator (Maxwellian molecules) in Theorem \ref{maintheorem} 
\begin{align*}
 Q(f,f)(t, x, v)=\int_{\mathbb{S}^{d-1}}\int_{\mathbb{R}^d}[f(t,x,u^*)f(t,x,v^*)-f(t,x,u)f(t,x,v)]\mathbf{b}\Big(\frac{v-u}{|v-u|}\cdot\omega\Big)\,dud\omega    
\end{align*}
under the assumption (which is known as Grad's cut-off assumption):
\begin{align*}
    \int_{\mathbb{S}^{d-1}}\mathbf{b}\Big(\frac{v-u}{|v-u|}\cdot\omega\Big )d\omega <\infty.
\end{align*}
Indeed, in this case, we would have
\begin{equation*}
 \begin{aligned}
 \tilde{Q}(\tilde{f},\tilde{f})(\xi)= \int_{\mathbb{S}^{d-1}}\tilde{f}(\xi^+) \tilde{f}(\xi^-)\mathbf{b}\Big(\frac{\xi}{|\xi|}\cdot\omega\Big )d\omega-\tilde{f}(\xi)\tilde{f}(0)\int_{\mathbb{S}^{d-1}}\mathbf{b}\Big(\frac{\xi}{|\xi|}\cdot\omega \Big)d\omega.
\end{aligned}   
\end{equation*}
Then, using theorem $1$ of \cite{ALONSO} in establishing a version of Lemma \ref{L_xiestimate} (with $\mathbf{b}$) together with the fact 
\begin{align*}
    \int_{\mathbb{S}^{d-1}}\mathbf{b}\Big(\frac{\xi}{|\xi|}\cdot\omega\Big )d\omega <\infty
\end{align*}
allows us to execute the same analysis to arrive at Theorem \ref{maintheorem}.   
\end{Remark}
 
\subsection{Organization of the paper} In Section \ref{sectStrichartz}, we prove the Strichartz estimate for the linear part of the equation \eqref{hyperbolicschrdingerIVP}, then in Section \ref{well-posedness}, we deal with the nonlinear estimates for loss and gain terms \eqref{lossandgainop} and prove the well-posedness assertion of Theorem \ref{maintheorem}.
\section{The Strichartz Estimate}\label{sectStrichartz}
In this section our goal is to prove the following Strichartz inequality. To begin with, we first set $\mathcal{M}=\mathbb{T}^d\times \mathbb {R}^d$ for $d\geq 1$. 
\begin{theorem}\label{stricharrtzboltzmann}
 Let $I\subset \mathbb{R}$ be an interval. Then, for all $M, N\geq 1$, we have
 \begin{align*}
     \Vert e^{it\nabla_{\xi}\cdot \nabla_x} P_N^xP_M^{\xi}\phi\Vert_{L^4(I\times \mathcal{M})}\leq C(I)\max\{M^d, (MN)^{d-1}\log N\}^{\frac{1}{4}} \Vert P_N^xP_M^{\xi}\phi\Vert_{L^2(\mathcal{M})}
 \end{align*}
 where $C(I)$ is a constant depending only on the measure of $I$.
\end{theorem}
The proof of Theorem \ref{stricharrtzboltzmann} is based on an argument as in \cite{Takaoka-Tzvetkov}. As a consequence of Theorem \ref{stricharrtzboltzmann}, we obtain the following corollary:   
\begin{Corollary}\label{stricharrtzboltzmanncor}
 Let $I\subset \mathbb{R}$ be an interval. Then, for all $n_0\in\mathbb{Z}^d$ and all $M, N\geq 1$, we have
 \begin{align}\label{corollorytostrichartz}
     \Vert e^{it\nabla_{\xi}\cdot \nabla_x} P_{n_0+(-N,N]^d}^xP_M^{\xi}\phi\Vert_{L^4(I\times \mathcal{M})}\leq C(I)\max\{M^d, (MN)^{d-1}\log N\}^{\frac{1}{4}} \Vert P_{n_0+(-N,N]^d}^xP_M^{\xi}\phi\Vert_{L^2(\mathcal{M})}
 \end{align}
 where $C(I)$ is as in Theorem \ref{stricharrtzboltzmann}.
\end{Corollary}
\begin{proof}
To shift the center of the frequency localization, we write
\begin{align*}
    x\cdot n-tv\cdot n=x\cdot n_0-tv\cdot n_0+x\cdot (n-n_0)-tv\cdot (n-n_0)
    \end{align*}
  and get
    \begin{equation}\label{e^itexpress}
  \begin{aligned}
   e^{it\nabla_{\xi}\cdot \nabla_x} &P_{n_0+(-N,N]^d}^xP_M^{\xi}\phi\,(t,x,\xi)\\&= \int_{\mathbb{R}^d} e^{i\xi\cdot v}\varphi_M(v) \sum_{n\in n_0+(-N,N]^d}e^{i(x\cdot n-tv\cdot n)}\widehat{\phi}(n,v)\,dv\\&=e^{ix\cdot n_0} \int_{\mathbb{R}^d} e^{i(\xi-tn_0)\cdot v}\varphi_M(v) \sum_{n\in n_0+(-N,N]^d}e^{i(x\cdot (n-n_0)-tv\cdot (n-n_0))}\widehat{\phi}(n,v)\,dv \\&=e^{ix\cdot n_0} \int_{\mathbb{R}^d} e^{i(\xi-tn_0)\cdot v}\varphi_M(v) \sum_{n\in (-N,N]^d}e^{i(x\cdot n-tv\cdot n)}\widehat{\phi}_0(n,v)\,dv\\&= e^{ix\cdot n_0}e^{it\nabla_{\xi}\cdot \nabla_x} P_{(-N,N]^d}^xP_M^{\xi}\phi_0\,(t,x,\xi-tn_0)
    \end{aligned}      
    \end{equation}
 where $\phi_0(x,\xi):=e^{-ix\cdot n_0}\phi(x,\xi)$. Likewise, we have
 \begin{align}\label{phiphi_0}
   P_{n_0+(-N,N]^d}^xP_M^{\xi}\phi\,(x,\xi)=e^{ix\cdot n_0}P_{(-N,N]^d}^xP_M^{\xi}\phi_0\,(x,\xi).  
 \end{align}
 Therefore, \eqref{corollorytostrichartz} follows from  \eqref{e^itexpress}, \eqref{phiphi_0}, and Theorem \ref{stricharrtzboltzmann}.
 
\end{proof}
Let $Q_K$ be the projection onto the frequencies $(\tau,n,v)$ such that $\langle \tau+v\cdot n\rangle\sim K$. To prove the Theorem \ref{stricharrtzboltzmann}, we begin by demonstrating the following modulation localised Strichartz inequality:
\begin{theorem}\label{L^2est}
    Let $u_1$ and $u_2$ be two functions defined on $\mathbb{R}\times\mathcal{M}$. Then, we have
 \begin{align*}
     \Vert Q_{K_1}P_N^xP_M^{\xi}u_1\,Q_{K_2}P_N^x&P_M^{\xi}u_2\Vert_{L^2(\mathbb{R}\times\mathcal{M})}\\&\lesssim K_1^{\frac{1}{2}}K_2^{\frac{1}{2}}\max\{M^d, (MN)^{d-1}\log N\} ^{\frac{1}{2}}\Vert P_N^xP_M^{\xi}u_1\Vert_{L^2(\mathbb{R}\times\mathcal{M})}\Vert P_N^xP_M^{\xi}u_2\Vert_{L^2(\mathbb{R}\times\mathcal{M})}.
 \end{align*}
 \end{theorem}
\begin{proof}
   For simplicity, let us assume that $u_j$ are localized functions with $Q_{K_j}P_N^xP_M^{\xi}u_j=u_j$, $j=1,2$. Then, by the Cauchy-Schwarz inequality
   \begin{align*}
       \Vert u_1u_2\Vert_{L^2}^2&=\int_{\mathbb{R}}\int_{\mathbb{R}^d}\sum_{n\in\mathbb{Z}^d}\Bigg|\int_{\mathbb{R}}\int_{\mathbb{R}^d}\sum_{n_1\in\mathbb{Z}^d}\widehat{u}_1(\tau_1,n_1,v_1)\widehat{u}_2(\tau-\tau_1,n-n_1,v-v_1)\,dv_1\,d\tau_1\Bigg|^2\,dv\,d\tau\\&\lesssim \sup_{\tau,n,v} |A_{\tau,n,v}|\Vert u_1\Vert_{L^2}^2\Vert u_2\Vert_{L^2}^2
   \end{align*}
   where 
\begin{align*}
    A_{\tau,n,v}=\{(\tau_1,n_1,v_1): |\tau_1+v_1\cdot n_1|\sim K_1, |\tau-\tau_1+(v-v_1)\cdot(n-n_1)|\sim K_2,\hspace{1cm}\\ |v_1|+|v-v_1|\lesssim M, |n_1|+|n-n_1|\lesssim N\}.
\end{align*}
For fixed $(n_1,v_1)$, the range of $\tau_1$ is $\min\{K_1,K_2\}$. Therefore, by the triangle inequality, we obtain
\begin{align}\label{A_tau}
 |A_{\tau,n,v}| \lesssim \min\{K_1,K_2\} |B_{\tau,n,v}| 
\end{align}
where
\begin{equation*}
 \begin{aligned}
B_{\tau,n,v}&=\{(n_1,v_1): |\tau+v_1\cdot n_1+(v-v_1)\cdot (n-n_1)|\lesssim K_1+K_2,\\& \hspace{5cm}|v_1|+|v-v_1|\lesssim M, |n_1|+|n-n_1|\lesssim N\}\\&=\{(n_1,v_1): |\tau+\frac{v\cdot n}{2}+2(v_1-\frac{v}{2})\cdot (n_1-\frac{n}{2})|\lesssim K_1+K_2,\\& \hspace{5cm}|v_1|+|v-v_1|\lesssim M, |n_1|+|n-n_1|\lesssim N\}.
\end{aligned}   
\end{equation*}
To finish the proof, we rely on following lemma:
\begin{lemma}\label{measure}
    Let $C_0\geq 0$\footnote{Here the value of $C_0$ may change from line to line in the proof, but for simplicity, we assume that $C_0\geq 0$ throughout. When $C_0 < 0$, similar argument also works.}, $M, N,K\geq 1$ be constants. Then,
    \begin{equation*}
    \begin{aligned}
        \sup_{\substack{|a|\lesssim M,|b|\lesssim N\\C_0\geq 0}}|\{(n,v)\in \mathbb{Z}^d\times \mathbb{R}^d: C_0\leq (v-a)\cdot (n-b)\leq C_0+K,\, &|v|\lesssim M, |n|\lesssim N\}|\\&\lesssim K\max\{M^d, (MN)^{d-1}\log N\}.
    \end{aligned}
    \end{equation*}
    where $|\cdot|$ denotes the product measure of the $d$-dimensional Lebesgue and counting measure. The implicit constant is independent of $C_0$.
\end{lemma}
\begin{proof}
 We may assume $K=1$, since for any $K\in \mathbb{N}$, we may write
 \begin{align*}
     \{&(n,v)\in \mathbb{Z}^d\times \mathbb{R}^d: C_0\leq (v-a)\cdot (n-b)\leq C_0+K,\, |v|\lesssim M, |n|\lesssim N\}\\&=\bigcup_{k=1}^K\{(n,v)\in \mathbb{Z}^d\times \mathbb{R}^d: C_0+k-1\leq (v-a)\cdot (n-b)\leq C_0+k,\, |v|\lesssim M, |n|\lesssim N\}.
 \end{align*}
 When $n=b$, we have the crude bound
 \begin{align*}
 \sup_{\substack{|a|\lesssim M,|b|\lesssim N\\C_0\geq 0}}|\{(n,v)\in \mathbb{Z}^d\times \mathbb{R}^d: C_0\leq (v-a)\cdot (n-b)\leq C_0+1,\, &|v|\lesssim M, |n|\lesssim N\}|\lesssim M^d.    
 \end{align*}
 Otherwise, we define 
 \begin{align*}
 V(b,M,N):=|\{(n,v)\in \mathbb{Z}^d\times \mathbb{R}^d: n\neq b,\, C_0\leq (v-a)\cdot (n-b)\leq C_0+1,\, &|v|\lesssim M, |n|\lesssim N\}|.    
 \end{align*}
 To estimate $V(b,M,N)$, let us express the variables coordinate-wise, that is, we shall write, e.g., $v=(v^{(1)}, v^{(2)},\cdots,v^{(d)})$, and express $n, a, b$ similarly as well. Assume that $n^{(1)}\neq b^{(1)}$. Fixing all coordinates except the first one and then using translation invariance to restrict $a^{(1)}=0$ and $b^{(1)}\in [0,1)$, we are able to obtain the following estimate
 \begin{align*}
   V(b,M,N)\lesssim (MN)^{d-1}V^{(1)}(b^{(1)},N),  
 \end{align*}
 where
 \begin{align*}
  V^{(1)}(b^{(1)},N):= |\{(n^{(1)},v^{(1)})\in \mathbb{Z}\times \mathbb{R}: n^{(1)}\neq b^{(1)},\, C_0\leq v^{(1)}(n^{(1)}-b^{(1)})\leq C_0+1,\, |n^{(1)}|\lesssim N\}|. 
 \end{align*}
 Therefore, it suffices to show that
 \begin{align*}
     \sup_{b\in[0,1), C_0\geq 0}V^{(1)}(b^{(1)},N)\lesssim \log N.
 \end{align*}
 To see this, we define
 \begin{align*}
    h(x)&=|\{(n^{(1)},v^{(1)})\in \mathbb{Z}\times \mathbb{R}: n^{(1)}\neq b^{(1)},\, 0\leq v^{(1)}(n^{(1)}-b^{(1)})\leq x,\, |n^{(1)}|\lesssim N\}| \\&=2\sum_{n^{(1)}=-\lfloor N\rfloor}^{\lfloor N\rfloor}\frac{x}{|n^{(1)}-b^{(1)}|}.
 \end{align*}
 Then, we get
 \begin{align*}
     V^{(1)}(b^{(1)},N)=h(C_0+1)-h(C_0)=2\sum_{n^{(1)}=-\lfloor N\rfloor}^{\lfloor N\rfloor}\frac{1}{|n^{(1)}-b^{(1)}|}\lesssim \log N.
 \end{align*}
\end{proof}
Therefore, by the Lemma \ref{measure}, we obtain
\begin{align*}
    |B_{\tau,n,v}|\lesssim (K_1+K_2)\max\{M^d, (MN)^{d-1}\log N\}\lesssim \max\{K_1,K_2\}\max\{M^d, (MN)^{d-1}\log N\}
\end{align*}
so that from \eqref{A_tau} we get
\begin{align*}
    |A_{\tau,n,v}|\lesssim K_1K_2\max\{M^d, (MN)^{d-1}\log N\}
\end{align*}
which implies the result.
\end{proof}
Theorem \ref{L^2est} leads to the following estimate.
\begin{Corollary}\label{L^4est}
For $b>\frac{1}{2}$, we have
\begin{align*}
    \Vert P_N^xP_M^{\xi}u\Vert_{L^4(\mathbb{R}\times \mathcal{M})}\lesssim \max\{M^d, (MN)^{d-1}\log N\}^{\frac{1}{4}} \Vert P_N^xP_M^{\xi}u\Vert_{X^{0,b}(\mathbb{R}\times \mathcal{M})}.
\end{align*}
\end{Corollary}
\begin{proof}
    For simplicity we shall write $P_N^xP_M^{\xi}u=u$. Then, by the Theorem \ref{L^2est}, we obtain
    \begin{align*}
        \Vert u\Vert_{L^4}^2=\Vert \sum_{K_1}Q_{K_1}u\sum_{K_2}Q_{K_2}u\Vert_{L^2}&\lesssim \sum_{K_1,K_2}\Vert Q_{K_1}u\,Q_{K_2}u\Vert_{L^2} \\&\lesssim \max\{M^d, (MN)^{d-1}\log N\}^{\frac{1}{2}}\sum_{K_1,K_2}K_1^{\frac{1}{2}-b}K_2^{\frac{1}{2}-b}\Vert u\Vert_{X^{0,b}}^2\\&\lesssim \max\{M^d, (MN)^{d-1}\log N\}^{\frac{1}{2}} \Vert u\Vert_{X^{0,b}}^2.
    \end{align*}
\end{proof}
Finally, we are ready to prove the Theorem \ref{stricharrtzboltzmann}.
\begin{proof}[Proof of Theorem \ref{stricharrtzboltzmann}]
  Let us write $P_M^{\xi}P_N^x\phi=\phi$ as before, and note that for $\delta>0$ we have
  \begin{equation}\label{psidelta}
   \begin{aligned}
      \Vert \psi_{\delta}\Vert_{H^b}\sim \Vert \psi_{\delta}\Vert_{L^2}+\Vert \psi_{\delta}\Vert_{\dot{H}^b}=\delta^{\frac{1}{2}}\Vert \psi\Vert_{L^2}+\delta^{\frac{1}{2}-b}\Vert \psi\Vert_{\dot{H}^b}\sim \delta^{\frac{1}{2}}+\delta^{\frac{1}{2}-b}.
  \end{aligned}   
  \end{equation}
 Therefore, from Corollary \ref{L^4est} and \eqref{psidelta}, we obtain
\begin{equation}\label{estimL^4delta}
 \begin{aligned}
    \Vert S(t)\phi\Vert_{L^4(|t|\leq \delta,(\xi,x)\in\mathcal{M})}&\sim \Vert \psi_{\delta}(t)S(t)\phi\Vert_{L^4(\mathbb{R}\times\mathcal{M})}\\&\lesssim \max\{M^d, (MN)^{d-1}\log N\}^{\frac{1}{4}}\Vert \psi_{\delta}(t)S(t)\phi\Vert_{X^{0,b}}\\&=\max\{M^d, (MN)^{d-1}\log N\}^{\frac{1}{4}}\Vert \psi_{\delta}\Vert_{H^b}\Vert \phi\Vert_{L^2_{x\xi}}\\&\lesssim \max\{M^d, (MN)^{d-1}\log N\}^{\frac{1}{4}}(\delta^{\frac{1}{2}}+\delta^{\frac{1}{2}-b})\Vert \phi\Vert_{L^2_{x\xi}},
\end{aligned}   
\end{equation}
where we have used the identity $\Vert f\Vert_{X^{0,b}}=\Vert S(-t)f\Vert_{H^b_tL^2_{x\xi}}$ in \eqref{estimL^4delta}. In the case of an arbitrary time interval $I=[\alpha,\beta]\subset \mathbb{R}$, we define $\Phi:=S(\frac{\alpha+\beta}{2})\phi$ so as to set 
\begin{align*}
 S(t)\phi=S(t-\frac{\alpha+\beta}{2})\Phi.  
\end{align*}
Since for $t\in I$, we have $-\frac{\beta-\alpha}{2}\leq t-\frac{\alpha+\beta}{2}\leq \frac{\beta-\alpha}{2}$ and $\norm{\Phi}_{L^2}=\norm{\phi}_{L^2}$, we proceed similarly as in \eqref{estimL^4delta} with $\delta=\frac{\beta-\alpha}{2}$ to obtain the desired estimate.  
\end{proof}
\begin{Remark}
    Using Theorem \ref{L^2est} we obtain
    \begin{align*}
     \Vert P_N^xP_M^{\xi}\tilde{f}\Vert_{L^4(I\times \mathcal{M})}\lesssim \max\{M^d, (MN)^{d-1}\log N\}^{\frac{1}{4}} \Vert P_N^xP_M^{\xi}\tilde{f}\Vert_{X^{0,b}(\mathbb{R}\times\mathcal{M})}\,\,\text{for}\,\, b>\frac{1}{2},  
    \end{align*}
    which will be useful in the next section while dealing with the nonlinear estimates regarding the equation \eqref{hyperbolicschrdingerIVP}. Also using Corollary \ref{stricharrtzboltzmanncor} along with the transference principle, we also obtain similar estimate as above for solutions acted by projection operators with arbitrary center of frequency localization. 
\end{Remark}
\section{Well-posedness}\label{well-posedness}
\subsection{Bilinear Estimates for Loss and Gain Terms}
\begin{lemma}\label{Q^-}
    For any $s>\frac{d}{2}-\frac{1}{4}$, $r>\frac{d}{2}$, we have
    \begin{align*}\label{lossest}
        \Vert \tilde{Q}^-(\tilde{f},\tilde{g})\Vert_{L^2_{T}H^s_xH^r_{\xi}}\lesssim T^{\frac{1}{4}}\Vert \tilde{f}\Vert_{X^{s,r,\frac{1}{2}+}}\Vert\tilde{g} \Vert_{X^{s,r,\frac{1}{2}+}}.
    \end{align*}
\end{lemma}
\begin{proof}
We have
  \begin{align*}
         \Vert \tilde{Q}^{-}(\tilde{f},\tilde{g})\Vert_{L^2_{T}H^s_xH^r_{\xi}}\lesssim \Vert \big(\langle \nabla_x \rangle^s\langle \nabla_{\xi} \rangle^r\tilde{f}(t,x,\xi)\big)\tilde{g}(t,x,0)\Vert_{L^2_{T}L^2_{x\xi}}+\Vert \big(\langle \nabla_{\xi} \rangle^r\tilde{f}(t,x,\xi)\big)\langle \nabla_x \rangle^s\tilde{g}(t,x,0)\Vert_{L^2_{T}L^2_{x\xi}}.
    \end{align*}
We start by estimating the first term above. 
    \begin{align*}
 \Vert \big(\langle \nabla_x \rangle^s\langle \nabla_{\xi} &\rangle^r\tilde{f}(t,x,\xi)\big)\tilde{g}(t,x,0)\Vert_{L^2_{T}L^2_{x\xi}}\\&\sim \Big(\sum_{N,M}\Vert P_{N,M}\big(\sum_{N_1,M_1}P_{N_1,M_1}\langle \nabla_{x} \rangle^s\langle \nabla_{\xi} \rangle^r\tilde{f}(t,x,\xi)\big)\big(\sum_{N_2}P_{N_2}\tilde{g}(t,x,0)\big)\Vert_{L^2_{T}L^2_{x\xi}}^2\Big)^{\frac{1}{2}}=:S_1.   
\end{align*}
\\{\bf Case A.} $N_1\gtrsim N_2$.\\
By Bernstein inequality, we have 
\begin{equation}\label{S_1}
  \begin{aligned} 
    S_1&\lesssim \Big(\sum_{N_1,M_1}\Vert P_{N_1,M_1}\langle \nabla_{x} \rangle^s\langle \nabla_{\xi} \rangle^r\tilde{f}(t,x,\xi)\sum_{N_2}P_{N_2}\tilde{g}(t,x,0)\Vert_{L^2_{T}L^2_{x\xi}}^2\Big)^{\frac{1}{2}} \\& \lesssim \Big(\sum_{N_1,M_1}\Big[\sum_{N_2}\Vert P_{N_1,M_1}\langle \nabla_{x} \rangle^s\langle \nabla_{\xi} \rangle^r\tilde{f}(t,x,\xi)P_{N_2}\tilde{g}(t,x,0)\Vert_{L^2_{T}L^2_{x\xi}}\Big]^2\Big)^{\frac{1}{2}}\\&\lesssim \Big(\sum_{N_1,M_1}\Big[\sum_{N_2}\Vert P_{N_1,M_1}\langle \nabla_{x} \rangle^s\langle \nabla_{\xi} \rangle^r\tilde{f}(t,x,\xi)\Vert_{L^{\infty}_{T}L^2_{x\xi}}\Vert \sum_{M_2}P_{N_2,M_2}\tilde{g}(t,x,\xi)\Vert_{L^2_{T}L^{\infty}_{x\xi}}\Big]^2\Big)^{\frac{1}{2}}
    \\&\lesssim T^{\frac{1}{4}}\Big(\sum_{N_1,M_1}\Big[\sum_{N_2, M_2}N_2^{\frac{d}{4}}M_2^{\frac{d}{4}}\Vert P_{N_1,M_1}\langle \nabla_{x} \rangle^s\langle \nabla_{\xi} \rangle^r\tilde{f}(t,x,\xi)\Vert_{L^{\infty}_{T}L^2_{x\xi}}\Vert P_{N_2,M_2}\tilde{g}(t,x,\xi)\Vert_{L^4_{T}L^4_{x\xi}}\Big]^2\Big)^{\frac{1}{2}}.
\end{aligned}
\end{equation}
In the case $N_2^{d-1}\log N_2\gg M_2$, by using the Strichartz estimate, the resulting bound in \eqref{S_1} can be estimated by
\begin{equation*}
    \begin{aligned}
        T^{\frac{1}{4}}\Big(\sum_{N_1,M_1}\Vert P_{N_1,M_1}\tilde{f}\Vert_{X^{s,r,\frac{1}{2}+}}^2\Big[\sum_{N_2, M_2}N_2^{\frac{d}{2}-\frac{1}{4}-s+}M_2^{\frac{d}{2}-\frac{1}{4}-r}\Vert P_{N_2,M_2}\tilde{g}\Vert_{X^{s,r,\frac{1}{2}+}}\Big]^2\Big)^{\frac{1}{2}}\lesssim T^{\frac{1}{4}}\Vert \tilde{f}\Vert_{X^{s,r,\frac{1}{2}+}}\Vert \tilde{g}\Vert_{X^{s,r,\frac{1}{2}+}}
    \end{aligned}
\end{equation*}
provided that $s,r>\frac{d}{2}-\frac{1}{4}$. While in the case $N_2^{d-1}\log N_2\lesssim M_2$, the application of Strichartz estimate to \eqref{S_1} yields
\begin{equation*}
    \begin{aligned}
        S_1 \lesssim T^{\frac{1}{4}}\Big(\sum_{N_1,M_1}\Vert P_{N_1,M_1}\tilde{f}\Vert_{X^{s,r,\frac{1}{2}+}}^2\Big[\sum_{N_2, M_2}N_2^{\frac{d}{4}-s}M_2^{\frac{d}{2}-r}\Vert P_{N_2,M_2}\tilde{g}\Vert_{X^{s,r,\frac{1}{2}+}}\Big]^2\Big)^{\frac{1}{2}}\lesssim T^{\frac{1}{4}}\Vert \tilde{f}\Vert_{X^{s,r,\frac{1}{2}+}}\Vert \tilde{g}\Vert_{X^{s,r,\frac{1}{2}+}}.
    \end{aligned}
\end{equation*}
for $s>\frac{d}{4}$ and $r>\frac{d}{2}$.
\\{\bf Case B.} $N_1\ll N_2$.\\
In this case, we have
\begin{equation}\label{S_1B}
  \begin{aligned} 
    S_1&\lesssim \Big(\sum_{N_2,M_1}\Vert \sum_{N_1}P_{N_1,M_1}\langle \nabla_{x} \rangle^s\langle \nabla_{\xi} \rangle^r\tilde{f}(t,x,\xi)P_{N_2}\tilde{g}(t,x,0)\Vert_{L^2_{T}L^2_{x\xi}}^2\Big)^{\frac{1}{2}} \\& \lesssim \Big(\sum_{N_2,M_1}\Big[\sum_{N_1}\Vert P_{N_1,M_1}\langle \nabla_{x} \rangle^s\langle \nabla_{\xi} \rangle^r\tilde{f}(t,x,\xi)P_{N_2}\tilde{g}(t,x,0)\Vert_{L^2_{T}L^2_{x\xi}}\Big]^2\Big)^{\frac{1}{2}}\\&\lesssim \Big(\sum_{N_2,M_1}\Big[\sum_{N_1}\Vert P_{N_1,M_1}\langle \nabla_{x} \rangle^s\langle \nabla_{\xi} \rangle^r\tilde{f}(t,x,\xi)\Vert_{L^{\infty}_{T}L^2_{x\xi}}\Vert \sum_{M_2}P_{N_2,M_2}\tilde{g}(t,x,\xi)\Vert_{L^2_{T}L^{\infty}_{x\xi}}\Big]^2\Big)^{\frac{1}{2}}
    \\&\lesssim T^{\frac{1}{4}}\Big(\sum_{N_2,M_1}\Big[\sum_{N_1, M_2}N_2^{\frac{d}{4}}M_2^{\frac{d}{4}}\Vert P_{N_1,M_1}\langle \nabla_{x} \rangle^s\langle \nabla_{\xi} \rangle^r\tilde{f}(t,x,\xi)\Vert_{L^{\infty}_{T}L^2_{x\xi}}\Vert P_{N_2,M_2}\tilde{g}(t,x,\xi)\Vert_{L^4_{T}L^4_{x\xi}}\Big]^2\Big)^{\frac{1}{2}}.
\end{aligned}
\end{equation}
When $N_2^{d-1}\log N_2\gg M_2$, by using the Strichartz estimate, the right side in \eqref{S_1B} can be majorized by
\begin{equation*}
    \begin{aligned}
        T^{\frac{1}{4}}\Big(\sum_{N_2,M_1}&\Big[\sum_{N_1, M_2}N_2^{\frac{d}{2}-\frac{1}{4}-s+}M_2^{\frac{d}{2}-\frac{1}{4}-r}\Vert P_{N_1,M_1}\tilde{f}\Vert_{X^{s,r,\frac{1}{2}+}}\Vert P_{N_2,M_2}\tilde{g}\Vert_{X^{s,r,\frac{1}{2}+}}\Big]^2\Big)^{\frac{1}{2}}\\&\lesssim T^{\frac{1}{4}}\Big(\sum_{N_2,M_1}\Vert P_{M_1}\tilde{f}\Vert_{X^{s,r,\frac{1}{2}+}}^2\Vert P_{N_2}\tilde{g}\Vert_{X^{s,r,\frac{1}{2}+}}^2\Big[\sum_{N_1, M_2}N_1^{\frac{d}{2}-\frac{1}{4}-s+}M_2^{\frac{d}{2}-\frac{1}{4}-r}\Big]^2\Big)^{\frac{1}{2}}\\&\lesssim T^{\frac{1}{4}}\Vert \tilde{f}\Vert_{X^{s,r,\frac{1}{2}+}}\Vert \tilde{g}\Vert_{X^{s,r,\frac{1}{2}+}}
    \end{aligned}
\end{equation*}
provided that $s,r>\frac{d}{2}-\frac{1}{4}$. If $N_2^{d-1}\log N_2\lesssim M_2$, then for $s>\frac{d}{4}$, $r>\frac{d}{2}$, we obtain
\begin{equation*}
    \begin{aligned}
        T^{\frac{1}{4}}\Big(\sum_{N_2,M_1}&\Big[\sum_{N_1, M_2}N_2^{\frac{d}{4}-s}M_2^{\frac{d}{2}-r}\Vert P_{N_1,M_1}\tilde{f}\Vert_{X^{s,r,\frac{1}{2}+}}\Vert P_{N_2,M_2}\tilde{g}\Vert_{X^{s,r,\frac{1}{2}+}}\Big]^2\Big)^{\frac{1}{2}}\lesssim T^{\frac{1}{4}}\Vert \tilde{f}\Vert_{X^{s,r,\frac{1}{2}+}}\Vert \tilde{g}\Vert_{X^{s,r,\frac{1}{2}+}}.
    \end{aligned}
\end{equation*}
As for the second term, we set
\begin{align*}
 \Vert \big(\langle \nabla_{\xi} &\rangle^r\tilde{f}(t,x,\xi)\big)\langle \nabla_x \rangle^s\tilde{g}(t,x,0)\Vert_{L^2_{T}L^2_{x\xi}}\\&\sim \Big(\sum_{N,M}\Vert P_{N,M}\big(\sum_{N_1,M_1}P_{N_1,M_1}\langle \nabla_{\xi} \rangle^r\tilde{f}(t,x,\xi)\big)\big(\sum_{N_2}P_{N_2}\langle \nabla_{x} \rangle^s\tilde{g}(t,x,0)\big)\Vert_{L^2_{T}L^2_{x\xi}}^2\Big)^{\frac{1}{2}}=:S_2.   
\end{align*}
\\{\bf Case A.} $N_1\gtrsim N_2$.\\
In this case, by Bernstein inequality we have the bound \begin{equation}\label{S}
  \begin{aligned} 
    S_2&\lesssim \Big(\sum_{N_1,M_1}\Vert P_{N_1,M_1}\langle \nabla_{\xi} \rangle^r\tilde{f}(t,x,\xi)\sum_{N_2\lesssim N_1}P_{N_2}\langle \nabla_{x} \rangle^s\tilde{g}(t,x,0)\Vert_{L^2_{T}L^2_{x\xi}}^2\Big)^{\frac{1}{2}} \\& \lesssim \Big(\sum_{N_1,M_1}\Big[\sum_{N_2\lesssim N_1}\Vert P_{N_1,M_1}\langle \nabla_{\xi} \rangle^r\tilde{f}(t,x,\xi)P_{N_2}\langle \nabla_{x} \rangle^s\tilde{g}(t,x,0)\Vert_{L^2_{T}L^2_{x\xi}}\Big]^2\Big)^{\frac{1}{2}}\\&\lesssim \Big(\sum_{N_1,M_1}\Big[\sum_{N_2\lesssim N_1}\Vert P_{N_1,M_1}\langle \nabla_{\xi} \rangle^r\tilde{f}(t,x,\xi)\Vert_{L^{\infty}_{T}L^4_xL^2_{\xi}}\Vert \sum_{M_2}P_{N_2,M_2}\langle \nabla_{x} \rangle^s\tilde{g}(t,x,\xi)\Vert_{L^2_{T}L^4_{x}L^{\infty}_{\xi}}\Big]^2\Big)^{\frac{1}{2}}
    \\&\lesssim T^{\frac{1}{4}}\Big(\sum_{N_1,M_1}\Big[\sum_{N_2\lesssim N_1}\sum_{M_2}N_1^{\frac{d}{4}}M_2^{\frac{d}{4}}\Vert P_{N_1,M_1}\langle \nabla_{\xi} \rangle^r\tilde{f}(t,x,\xi)\Vert_{L^{\infty}_{T}L^2_{x\xi}}\Vert P_{N_2,M_2}\langle \nabla_{x} \rangle^s\tilde{g}(t,x,\xi)\Vert_{L^4_{T}L^4_{x\xi}}\Big]^2\Big)^{\frac{1}{2}}.
\end{aligned}
\end{equation}
Next, we consider the following regions in order to apply the Strichartz estimate to the final bound appeared in \eqref{S}.
\\{\bf Case A.1.} $N_2^{d-1}\log N_2\gg M_2$.\\
In this region, using the Strichartz estimate, we estimate the sum in the right hand side of \eqref{S} by
\begin{equation*}
    \begin{aligned}
 \Big(\sum_{N_1,M_1}&N_1^{\frac{d}{2}-2s}\Vert P_{N_1,M_1}\tilde{f}\Vert_{X^{s,r,\frac{1}{2}+}}^2\Big[\sum_{N_2\lesssim N_1}\sum_{M_2}N_2^{\frac{d-1}{4}+}M_2^{\frac{d}{2}-\frac{1}{4}-r}\Vert P_{N_2,M_2}\tilde{g}\Vert_{X^{s,r,\frac{1}{2}+}}\Big]^2\Big)^{\frac{1}{2}}\\&\lesssim \Big(\sum_{N_1,M_1}N_1^{d-\frac{1}{2}-2s+}\Vert P_{N_1,M_1}\tilde{f}\Vert_{X^{s,r,\frac{1}{2}+}}^2\Big)^{\frac{1}{2}} \Vert \tilde{g}\Vert_{X^{s,r,\frac{1}{2}+}}\lesssim \Vert \tilde{f}\Vert_{X^{s,r,\frac{1}{2}+}}\Vert \tilde{g}\Vert_{X^{s,r,\frac{1}{2}+}}      
    \end{aligned}
\end{equation*}
provided that $s, r>\frac{d}{2}-\frac{1}{4}$.
\\{\bf Case A.2.} $N_2^{d-1}\log N_2\lesssim M_2$.\\
Using the Strichartz estimate in relation to this case, the sum in the right hand side of \eqref{S} can be controlled by
\begin{equation*}
    \begin{aligned}
 \Big(\sum_{N_1,M_1}&\Vert P_{N_1,M_1}\tilde{f}\Vert_{X^{s,r,\frac{1}{2}+}}^2\Big[\sum_{N_2,M_2}N_2^{\frac{d}{4}-s}M_2^{\frac{d}{2}-r}\Vert P_{N_2,M_2}\tilde{g}\Vert_{X^{s,r,\frac{1}{2}+}}\Big]^2\Big)^{\frac{1}{2}} \\&\lesssim \Big(\sum_{N_1,M_1}\Vert P_{N_1,M_1}\tilde{f}\Vert_{X^{s,r,\frac{1}{2}+}}^2\sum_{N_2,M_2}\Vert P_{N_2,M_2}\tilde{g}\Vert_{X^{s,r,\frac{1}{2}+}}^2\Big)^{\frac{1}{2}}\sim \Vert \tilde{f}\Vert_{X^{s,r,\frac{1}{2}+}}\Vert \tilde{g}\Vert_{X^{s,r,\frac{1}{2}+}}    
    \end{aligned}
\end{equation*}
provided that $s>\frac{d}{4}$ and $r>\frac{d}{2}$.
\\{\bf Case B.} $N_1\ll N_2$.\\
In this case, the idea is to partition the annulus ${|n_2| \sim N_2}$ into squares of side length $\sim N_1$. Then, by almost orthogonality, we have
\begin{equation}\label{sumloss}
 \begin{aligned} 
    S_2&\lesssim \Big(\sum_{N_2,M_1}\Vert \Big(\sum_{N_1}P_{N_1,M_1}\langle \nabla_{\xi} \rangle^r\tilde{f}(t,x,\xi)\Big)P_{N_2}\langle \nabla_{x} \rangle^s\tilde{g}(t,x,0)\Vert_{L^2_{T}L^2_{x\xi}}^2\Big)^{\frac{1}{2}} \\&\lesssim \Big(\sum_{N_2,M_1} \Big[\sum_{N_1}\Vert\sum_{a\in\mathbb{Z}^d}P_{a+(-N_1,N_1]^d}[P_{N_1,M_1}\langle \nabla_{\xi} \rangle^r\tilde{f}(t,x,\xi)P_{N_2}\langle \nabla_{x} \rangle^s\tilde{g}(t,x,0)]\Vert_{L^2_{T}L^2_{x\xi}}\Big]^2\Big)^{\frac{1}{2}}\\&\sim \Big(\sum_{N_2,M_1} \Big[\sum_{N_1}\Big(\sum_{a\in\mathbb{Z}^d}\Vert P_{N_1,M_1}\langle \nabla_{\xi} \rangle^r\tilde{f}(t,x,\xi)P_{a+(-N_1,N_1]^d}P_{N_2}\langle \nabla_{x} \rangle^s\tilde{g}(t,x,0)\Vert_{L^2_{T}L^2_{x\xi}}^2\Big)^{\frac{1}{2}}\Big]^2\Big)^{\frac{1}{2}}.
\end{aligned}
\end{equation}
We first estimate the norm inside the sum in \eqref{sumloss} as follows
\begin{equation}\label{strichartzloss}
\begin{aligned}
\Vert P_{N_1,M_1}&\langle \nabla_{\xi} \rangle^r\tilde{f}(t,x,\xi)P_{a+(-N_1,N_1]^d}P_{N_2}\langle \nabla_{x} \rangle^s\tilde{g}(t,x,0)\Vert_{L^2_{T}L^2_{x\xi}}
\\&\lesssim
    \Vert P_{N_1,M_1}\langle \nabla_{\xi} \rangle^r\tilde{f}(t,x,\xi)\Vert_{L^{\infty}_{T}L^4_xL^2_{\xi}}\Vert P_{a+(-N_1,N_1]^d}P_{N_2}\langle \nabla_{x} \rangle^s\tilde{g}(t,x,\xi)\Vert_{L^2_{T}L^4_{x}L^{\infty}_{\xi}}\\&\lesssim N_1^{\frac{d}{4}}\Vert P_{N_1,M_1}\langle \nabla_{\xi} \rangle^r\tilde{f}(t,x,\xi)\Vert_{L^{\infty}_{T}L^2_{x\xi}}\Vert \sum_{M_2}P_{a+(-N_1,N_1]^d}P_{N_2,M_2}\langle \nabla_{x} \rangle^s\tilde{g}(t,x,\xi)\Vert_{L^2_{T}L^4_{x}L^{\infty}_{\xi}}\\&\lesssim T^{\frac{1}{4}}\sum_{M_2}N_1^{\frac{d}{4}}M_2^{\frac{d}{4}}\Vert P_{N_1,M_1}\langle \nabla_{\xi} \rangle^r\tilde{f}(t,x,\xi)\Vert_{L^{\infty}_{T}L^2_{x\xi}}\Vert P_{a+(-N_1,N_1]^d}P_{N_2,M_2}\langle \nabla_{x} \rangle^s\tilde{g}(t,x,\xi)\Vert_{L^4_{T}L^4_{x\xi}},
\end{aligned}
\end{equation}
then, proceed as above.
\\{\bf Case B.1.} $N_1^{d-1}\log N_1\gg M_2$.\\
We apply Strichartz estimate to the $L^4$ norm in \eqref{strichartzloss} to obtain 
\begin{align*}
\Vert P_{N_1,M_1}&\langle \nabla_{\xi} \rangle^r\tilde{f}(t,x,\xi)P_{a+(-N_1,N_1]^d}P_{N_2}\langle \nabla_{x} \rangle^s\tilde{g}(t,x,0)\Vert_{L^2_{T}L^2_{x\xi}}
\\&\lesssim
     T^{\frac{1}{4}}\sum_{M_2}N_1^{\frac{d}{2}-\frac{1}{4}-s+}M_2^{\frac{d}{2}-\frac{1}{4}-r}\Vert P_{N_1,M_1}\tilde{f}\Vert_{X^{s,r,\frac{1}{2}+}}\Vert P_{a+(-N_1,N_1]^d}P_{N_2,M_2}\tilde{g}\Vert_{X^{s,r,\frac{1}{2}+}}\\&\lesssim T^{\frac{1}{4}}N_1^{\frac{d}{2}-\frac{1}{4}-s+}\Vert P_{N_1,M_1}\tilde{f}\Vert_{X^{s,r,\frac{1}{2}+}}\Vert P_{a+(-N_1,N_1]^d}P_{N_2}\tilde{g}\Vert_{X^{s,r,\frac{1}{2}+}}
\end{align*}
for $r>\frac{d}{2}-\frac{1}{4}$, which leads to the bound below for the right hand side of \eqref{sumloss}
\begin{equation*}
    \begin{aligned}
     T^{\frac{1}{4}}\Big(\sum_{N_2,M_1}\Big[\sum_{N_1}N_1^{\frac{d}{2}-\frac{1}{4}-s+}&\Vert P_{N_1,M_1}\tilde{f}\Vert_{X^{s,r,\frac{1}{2}+}}\Big(\sum_{a\in \mathbb{Z}^d}\Vert P_{a+(-N_1,N_1]^d}P_{N_2}\tilde{g}\Vert_{X^{s,r,\frac{1}{2}+}}^2\Big)^{\frac{1}{2}}\Big]^2\Big)^{\frac{1}{2}} \\&\lesssim T^{\frac{1}{4}}\Big(\sum_{N_2,M_1}\Vert P_{M_1}\tilde{f}\Vert_{X^{s,r,\frac{1}{2}+}}^2\Vert P_{N_2}\tilde{g}\Vert_{X^{s,r,\frac{1}{2}+}}^2\Big)^{\frac{1}{2}}\\&\sim  T^{\frac{1}{4}}\Vert \tilde{f}\Vert_{X^{s,r,\frac{1}{2}+}}\Vert \tilde{g}\Vert_{X^{s,r,\frac{1}{2}+}} 
    \end{aligned}
\end{equation*}
as long as $s,r>\frac{d}{2}-\frac{1}{4}$.
\\{\bf Case B.2.} $N_1^{d-1}\log N_1\lesssim M_2$.\\
Applying Strichartz estimate in this region to the $L^4$ norm in \eqref{strichartzloss} gives that 
\begin{align*}
\Vert P_{N_1,M_1}&\langle \nabla_{\xi} \rangle^r\tilde{f}(t,x,\xi)P_{a+(-N_1,N_1]^d}P_{N_2}\langle \nabla_{x} \rangle^s\tilde{g}(t,x,0)\Vert_{L^2_{T}L^2_{x\xi}}
\\&\lesssim
     T^{\frac{1}{4}}N_1^{\frac{d}{4}-s}\Vert P_{N_1,M_1}\tilde{f}\Vert_{X^{s,r,\frac{1}{2}+}}\sum_{M_2}M_2^{\frac{d}{2}-r}\Vert P_{a+(-N_1,N_1]^d}P_{N_2,M_2}\tilde{g}\Vert_{X^{s,r,\frac{1}{2}+}}\\&\lesssim T^{\frac{1}{4}}N_1^{\frac{d}{4}-s}\Vert P_{N_1,M_1}\tilde{f}\Vert_{X^{s,r,\frac{1}{2}+}}\Vert P_{a+(-N_1,N_1]^d}P_{N_2}\tilde{g}\Vert_{X^{s,r,\frac{1}{2}+}}
\end{align*}
for $r>\frac{d}{2}$. As a result, proceeding as in the previous case by substituting this into \eqref{sumloss} leads to the desired estimate for $s>\frac{d}{4}$, $r>\frac{d}{2}$. 
\end{proof}
As regards to the proof of the bilinear estimate for the gain term, we rely on the following lemma, which is basically a Hölder's inequality in the $\xi$-direction concerning the term $\tilde{Q}^+$. For the proof, see \cite{Chen-Holmer,ALONSO}.
\begin{lemma}\label{L_xiestimate}
    Let $p,q>\frac{3}{2}$ with $\frac{1}{p}+\frac{1}{q}=\frac{1}{r}$. Then, we have
    \begin{align*}
    \Vert P_{N}^xP_M^{\xi}\tilde{Q}^{+}(P_{N_1}^xP_{M_1}^{\xi}\tilde{f},P_{N_2}^xP_{M_2}^{\xi}\tilde{g})\Vert_{L^r_{\xi}} \lesssim \Vert P_{N_1}^xP_{M_1}^{\xi}\tilde{f}\Vert_{L^p_{\xi}}\Vert P_{N_2}^xP_{M_2}^{\xi}\tilde{g}\Vert_{L^q_{\xi}}.    
    \end{align*}
\end{lemma}
\begin{lemma}\label{Q^+}
    For any $s>\frac{d}{2}-\frac{1}{4}$, $r>\frac{d}{2}$, we have
    \begin{align}\label{gainest}
        \Vert \tilde{Q}^+(\tilde{f},\tilde{g})\Vert_{L^2_{T}H^s_xH^r_{\xi}}\lesssim T^{\frac{1}{4}}\Vert \tilde{f}\Vert_{X^{s,r,\frac{1}{2}+}}\Vert\tilde{g} \Vert_{X^{s,r,\frac{1}{2}+}}.
    \end{align}
\end{lemma}
\begin{proof}
    It suffices to estimate
\begin{equation}\label{gainestdecomposed}
      \begin{aligned}
     \Vert \langle \nabla_x \rangle^s\langle \nabla_{\xi} \rangle^r\tilde{Q}^{+}(\tilde{f},\tilde{g})\Vert_{L^2_{T}L^2_{x\xi}}&\sim \Big(\sum_{N,M}N^{2s}M^{2r}\Vert P_{N,M}\int_{\mathbb{S}^{d-1}}\sum_{N_1,M_1}P_{N_1,M_1}\tilde{f}\sum_{N_2,M_2}P_{N_2,M_2}\tilde{g}\,\text{d}\omega\Vert_{L^2_{T}L^2_{x\xi}}^2\Big)^{\frac{1}{2}}\\&\lesssim \Big(\sum_{N,M}N^{2s}M^{2r}\Big[\sum_{\substack{N_1,M_1\\N_2,M_2}}\Vert P_{N,M}\tilde{Q}^{+}(P_{N_1,M_1}\tilde{f},P_{N_2,M_2}\tilde{g})\Vert_{L^2_{T}L^2_{x\xi}}\Big]^2\Big)^{\frac{1}{2}}.
    \end{aligned}
  \end{equation}
  Note that we make use of the constraint $N\lesssim \max\{N_1,N_2\}$, since 
  \begin{align*}
      P_N^x\tilde{Q}^+(P_{N_1}^x\tilde{f},P_{N_2}^x\tilde{g})(t,x,\xi)&=P_N^x\int_{\mathbb{S}^{d-1}}P_{N_1}^x\tilde{f}(t,x,\xi^+)P_{N_2}^x\tilde{g}(t,x,\xi^-)\,\text{d}\omega\\&=\int_{\mathbb{S}^{d-1}}P_N^x\big(P_{N_1}^x\tilde{f}(t,x,\xi^+)P_{N_2}^x\tilde{g}(t,x,\xi^-)\big)\,\text{d}\omega=0
  \end{align*}
  whenever $N\gg \max\{N_1,N_2\}$. We may take advantage of the constraint $M\lesssim \max\{M_1,M_2\}$ as well. To see why, first note, by uncovering the notations, that
\begin{equation}\label{supportcond}
  \begin{aligned}
\mathcal{F}_{\xi}\big(P_M^{\xi}\tilde{Q}^+(P_{M_1}^{\xi}\tilde{f},P_{M_2}^{\xi}\tilde{g})\big)(v)&= \mathcal{F}_{\xi}\circ \mathcal{F}_{v}^{-1}\big(\varphi_M(v)\mathcal{F}_{\xi}\big(\tilde{Q}^+(P_{M_1}^{\xi}\tilde{f},P_{M_2}^{\xi}\tilde{g})\big)\big)(v)\\&=\varphi_M(v)\mathcal{F}_{\xi}\circ \mathcal{F}_{v}^{-1}\big(Q^+(\varphi_{M_1}f,\varphi_{M_2}g)\big)(v)\\&=\varphi_M(v)Q^+(\varphi_{M_1}f,\varphi_{M_2}g)(v)\\&=\varphi_M(v)\int_{\mathbb{S}^{d-1}}\int_{\mathbb{R}^d}(\varphi_{M_1}f)(v^*)(\varphi_{M_2}g)(u^*)\,\text{d}u\text{d}\omega.
  \end{aligned}
  \end{equation}
   Due to the energy conservation \eqref{conservation}, we have $|v|^2\leq |u^*|^2+|v^*|^2$, hence for $v\in \supp{\varphi_M}$, this inequality gives rise to the following  \begin{align}\label{u^*orv^*}
    |u^*|\gtrsim M\,\,\text{or}\,\,|v^*|\gtrsim M.   
   \end{align} Thus, if it is the case that $M\gg \max\{M_1,M_2\}$, then \eqref{u^*orv^*} implies that $(\varphi_{M_1}f)(v^*)=0$ or $(\varphi_{M_2}g)(u^*)=0$ for $v\in\supp{\varphi_M}$, hence the integral in \eqref{supportcond} vanishes, that is, $P_M^{\xi}\tilde{Q}^+(P_{M_1}^{\xi}\tilde{f},P_{M_2}^{\xi}\tilde{g})=0$ for this choice. As a result, in what follows, we assume both the restrictions $N\lesssim \max\{N_1,N_2\}$ and $M\lesssim \max\{M_1,M_2\}$, and consider the regions below.
  \\{\bf Case A.} $N_1\geq N_2$, $M_1\geq M_2$.\\
  In this case, the right hand side of \eqref{gainestdecomposed} is estimated by
  \begin{align}\label{gainlittle}
\Big(\sum_{N,M}\Big[\sum_{\substack{N_1\gtrsim N,M_1\gtrsim M\\N_2,M_2}}\Big(\frac{N}{N_1}\Big)^s\Big(\frac{M}{M_1}\Big)^r\Vert P_{N,M}\tilde{Q}^{+}(P_{N_1,M_1}\langle \nabla_x\rangle^s\langle \nabla_{\xi}\rangle^r\tilde{f},P_{N_2,M_2}\tilde{g})\Vert_{L^2_{T}L^2_{x\xi}}\Big]^2\Big)^{\frac{1}{2}}.
  \end{align}
  Next, we estimate the $L^2$ norm above in the following regions.
\\{\bf Case A.1.} $M_2\leq N_2^{d-1}\log N_2$.\\
From the Lemma \ref{L_xiestimate}, Hölder, Bernstein, and Strichartz inequalities, it follows that
\begin{equation}\label{strica1}
  \begin{aligned}
 \Vert P_{N,M}\tilde{Q}^{+}(P_{N_1,M_1}\langle \nabla_x\rangle^s\langle \nabla_{\xi}\rangle^r\tilde{f},P_{N_2,M_2}\tilde{g})&\Vert_{L^2_{T}L^2_{x\xi}}\lesssim \Vert P_{N_1,M_1}\langle \nabla_x\rangle^s\langle \nabla_{\xi}\rangle^r\tilde{f}\Vert_{L^{\infty}_{T}L^2_{x\xi}} \Vert P_{N_2,M_2}\tilde{g}\Vert_{L^{2}_{T}L^{\infty}_{x\xi}}\\&\lesssim T^{\frac{1}{4}}N_2^{\frac{d}{4}}M_2^{\frac{d}{4}}\Vert P_{N_1,M_1}\tilde{f}\Vert_{X^{s,r,\frac{1}{2}+}}\Vert P_{N_2,M_2}\tilde{g}\Vert_{L^{4}_{T}L^{4}_{x\xi}}\\&\lesssim T^{\frac{1}{4}}N_2^{\frac{d}{2}-\frac{1}{4}-s+}M_2^{\frac{d}{2}-\frac{1}{4}-r}\Vert P_{N_1,M_1}\tilde{f}\Vert_{X^{s,r,\frac{1}{2}+}}\Vert P_{N_2,M_2}\tilde{g}\Vert_{{X^{s,r,\frac{1}{2}+}}}.    
  \end{aligned}  
\end{equation}
Therefore, upon substituting \eqref{strica1}, the sum in \eqref{gainlittle} is estimated by
\begin{equation*}
    \begin{aligned}
   T^{\frac{1}{4}}\Big(\sum_{N,M}&\Big[\sum_{\substack{N_1\gtrsim N\\M_1\gtrsim M}}\Big(\frac{N}{N_1}\Big)^s\Big(\frac{M}{M_1}\Big)^r\Vert P_{N_1,M_1}\tilde{f}\Vert_{X^{s,r,\frac{1}{2}+}}\sum_{N_2,M_2}N_2^{\frac{d}{2}-\frac{1}{4}-s+}M_2^{\frac{d}{2}-\frac{1}{4}-r}\Big]^2\Big)^{\frac{1}{2}} \Vert\tilde{g} \Vert_{X^{s,r,\frac{1}{2}+}}\\&\lesssim T^{\frac{1}{4}}\Big(\sum_{N,M}\Big[\sum_{\substack{N_1\gtrsim N\\M_1\gtrsim M}}\Big(\frac{N}{N_1}\Big)^s\Big(\frac{M}{M_1}\Big)^r\Vert P_{N_1,M_1}\tilde{f}\Vert_{X^{s,r,\frac{1}{2}+}}\Big]^2\Big)^{\frac{1}{2}} \Vert\tilde{g} \Vert_{X^{s,r,\frac{1}{2}+}}=:T^{\frac{1}{4}}S\Vert\tilde{g} \Vert_{X^{s,r,\frac{1}{2}+}}   
    \end{aligned}
\end{equation*}
provided that $s,r>\frac{d}{2}-\frac{1}{4}$. We write the following and proceed as follows:
\begin{equation*}
\begin{aligned}
 S =:\Big(\sum_{N,M}\Big[\sum_{\substack{N_1\gtrsim N\\M_1\gtrsim M}}K\Big(\frac{N}{N_1},\frac{M}{M_1}\Big)\beta_{N_1,M_1}\Big]^2\Big)^{\frac{1}{2}}
 \end{aligned}
\end{equation*}
where $K\big(\frac{N}{N_1},\frac{M}{M_1} \big)=\big(\frac{N}{N_1}\big)^s\big(\frac{M}{M_1}\big)^r$, and $\beta_{N_1,M_1}=\Vert P_{N_1,M_1}\tilde{f}\Vert_{X^{s,r,\frac{1}{2}+}}$. Letting $N=2^n$, $N_1=2^{n_1}$, $M=2^{m}$, and $M_1=2^{m_1}$ for $n, m, n_1, m_1\in\mathbb{Z}$ yields $K\big(\frac{N}{N_1},\frac{M}{M_1}\big)\sim 2^{-|n-n_1|-|m-m_1|}$ in this region. Therefore, by Young's inequality, we obtain
\begin{align*}
    S\sim \Big(\sum_{m,n}\Big[\sum_{m_1,n_1}2^{-|n-n_1|-|m-m_1|}\beta_{2^{n_1},2^{m_1}}\Big]^2\Big)^{\frac{1}{2}}=\Vert 2^{-|n|-|m|}*\beta_{2^{n},2^m}\Vert_{\ell^2_{m,n}}\lesssim \Vert\beta_{2^{n},2^m}\Vert_{\ell^2_{m,n}}\sim \Vert \tilde{f}\Vert_{X^{s,r,\frac{1}{2}+}}. 
\end{align*}
\\{\bf Case A.2.} $M_2\geq N_2^{d-1}\log N_2$.\\
By the Lemma \ref{L_xiestimate}, Hölder, Bernstein, and Strichartz inequalities, we obtain
\begin{equation*}
  \begin{aligned}
 \Vert P_{N,M}\tilde{Q}^{+}(P_{N_1,M_1}\langle \nabla_x\rangle^s\langle \nabla_{\xi}\rangle^r\tilde{f},P_{N_2,M_2}\tilde{g})\Vert_{L^2_{T}L^2_{x\xi}}\lesssim T^{\frac{1}{4}} N_2^{\frac{d}{4}-s}M_2^{\frac{d}{2}-r}\Vert P_{N_1,M_1}\tilde{f}\Vert_{X^{s,r,\frac{1}{2}+}}\Vert P_{N_2,M_2}\tilde{g}\Vert_{{X^{s,r,\frac{1}{2}+}}}   
  \end{aligned}  
\end{equation*}
which implies for $s>\frac{d}{4},\, r>\frac{d}{2}$:
 \begin{equation*}
    \begin{aligned}
\eqref{gainlittle}&\lesssim T^{\frac{1}{4}}\Big(\sum_{N,M}\Big[\sum_{\substack{N_1\gtrsim N\\M_1\gtrsim M}}\Big(\frac{N}{N_1}\Big)^s\Big(\frac{M}{M_1}\Big)^r\Vert P_{N_1,M_1}\tilde{f}\Vert_{X^{s,r,\frac{1}{2}+}}\sum_{N_2,M_2}N_2^{\frac{d}{4}-s}M_2^{\frac{d}{2}-r}\Big]^2\Big)^{\frac{1}{2}} \Vert\tilde{g} \Vert_{X^{s,r,\frac{1}{2}+}}\\&\lesssim T^{\frac{1}{4}}\Big(\sum_{N,M}\Big[\sum_{\substack{N_1\gtrsim N\\M_1\gtrsim M}}\Big(\frac{N}{N_1}\Big)^s\Big(\frac{M}{M_1}\Big)^r\Vert P_{N_1,M_1}\tilde{f}\Vert_{X^{s,r,\frac{1}{2}+}}\Big]^2\Big)^{\frac{1}{2}} \Vert\tilde{g} \Vert_{X^{s,r,\frac{1}{2}+}}\\&\lesssim T^{\frac{1}{4}}\Vert\tilde{f} \Vert_{X^{s,r,\frac{1}{2}+}}\Vert\tilde{g} \Vert_{X^{s,r,\frac{1}{2}+}}.   
    \end{aligned}
\end{equation*}
\\{\bf Case B.} $N_2\geq N_1$, $M_2\geq M_1$.\\
This case can be dealt in a similar manner as in the Case A., simply by the roles of $\tilde{f}$ and $\tilde{g}$ are switched.
\\{\bf Case C.} $N_1\geq N_2$, $M_2\geq M_1$.\\
In this region, by almost orthogonality, we obtain
\begin{equation}\label{almostorthgain}
\begin{aligned}
 \text{RHS of}&\,\, \eqref{gainestdecomposed}\\&\lesssim \Big(\sum_{N,M}N^{2s}M^{2r}\Big[\sum_{\substack{N_1\gtrsim N, M_2\gtrsim M\\N_2, M_1}}\Big(\sum_{a\in\mathbb{Z}^d}\Vert P_{N,M}\tilde{Q}^+(P_{a+(-N_2,N_2]^d}P_{N_1,M_1}\tilde{f},P_{N_2,M_2}\tilde{g})\Vert_{L^2_{T}L^2_{x\xi}}^2\Big)^{\frac{1}{2}}\Big]^2\Big)^{\frac{1}{2}}.    
\end{aligned}
\end{equation}
Then, we consider the cases below.
\\{\bf Case C.1.} $M_1\leq N_2^{d-1}\log N_2$.\\
The Lemma \ref{L_xiestimate}, Hölder, Bernstein, and Strichartz inequalities lead to
\begin{align*}
    \Vert P_{N,M}\tilde{Q}^+&(P_{a+(-N_2,N_2]^d}P_{N_1,M_1}\tilde{f},P_{N_2,M_2}\tilde{g})\Vert_{L^2_{T}L^2_{x\xi}}\\&\lesssim \Vert P_{a+(-N_2,N_2]^d}P_{N_1,M_1}\tilde{f}\Vert_{L^{2}_{T}L^4_xL^{\infty}_{\xi}}\Vert P_{N_2,M_2}\tilde{g}\Vert_{L^{\infty}_{T}L^4_xL^{2}_{\xi}}\\&\lesssim T^{\frac{1}{4}}N_2^{\frac{d}{4}}M_1^{\frac{d}{4}} \Vert P_{a+(-N_2,N_2]^d}P_{N_1,M_1}\tilde{f}\Vert_{L^{4}_{T}L^4_{x\xi}}\Vert P_{N_2,M_2}\tilde{g}\Vert_{L^{\infty}_{T}L^{2}_{x\xi}}\\&\lesssim T^{\frac{1}{4}}N_2^{\frac{d}{2}-\frac{1}{4}+}M_1^{\frac{d}{2}-\frac{1}{4}}\Vert P_{a+(-N_2,N_2]^d}P_{N_1,M_1}\tilde{f}\Vert_{X^{0,\frac{1}{2}+}}\Vert P_{N_2,M_2}\tilde{g}\Vert_{X^{0,\frac{1}{2}+}}
\end{align*}
which implies that
\begin{equation*}
    \begin{aligned}
 &\text{RHS of}\,\, \eqref{almostorthgain}\\&\lesssim T^{\frac{1}{4}}\Big(\sum_{N,M}N^{2s}M^{2r}\Big[\sum_{\substack{N_1\gtrsim N, M_2\gtrsim M\\N_2, M_1}}N_2^{\frac{d}{2}-\frac{1}{4}+}M_1^{\frac{d}{2}-\frac{1}{4}}\Vert P_{N_2,M_2}\tilde{g}\Vert_{X^{0,\frac{1}{2}+}}\\&\hspace{6.2cm}\times\Big(\sum_{a\in\mathbb{Z}^d}\Vert P_{a+(-N_2,N_2]^d}P_{N_1,M_1}\tilde{f}\Vert_{X^{0,\frac{1}{2}+}}^2\Big)^{\frac{1}{2}}\Big]^2\Big)^{\frac{1}{2}}\\&\sim T^{\frac{1}{4}}\Big(\sum_{N,M}N^{2s}M^{2r}\Big[\sum_{\substack{N_1\gtrsim N, M_2\gtrsim M\\N_2, M_1}}N_2^{\frac{d}{2}-\frac{1}{4}+}M_1^{\frac{d}{2}-\frac{1}{4}}\Vert P_{N_2,M_2}\tilde{g}\Vert_{X^{0,\frac{1}{2}+}}\Vert P_{N_1,M_1}\tilde{f}\Vert_{X^{0,\frac{1}{2}+}}\Big]^2\Big)^{\frac{1}{2}}\\&\sim T^{\frac{1}{4}}\Big(\sum_{N,M}\Big[\sum_{\substack{N_1\gtrsim N, M_2\gtrsim M\\N_2, M_1}}\Big(\frac{N}{N_1}\Big)^s\Big(\frac{M}{M_2}\Big)^rN_2^{\frac{d}{2}-\frac{1}{4}-s+}M_1^{\frac{d}{2}-\frac{1}{4}-r}\Vert P_{N_1,M_1}\tilde{f}\Vert_{X^{s,r,\frac{1}{2}+}}\Vert P_{N_2,M_2}\tilde{g}\Vert_{X^{s,r,\frac{1}{2}+}}\Big]^2\Big)^{\frac{1}{2}}\\&\lesssim T^{\frac{1}{4}}\Big(\sum_{N,M}\Big[\sum_{\substack{N_1\gtrsim N\\ M_2\gtrsim M}}\Big(\frac{N}{N_1}\Big)^s\Big(\frac{M}{M_2}\Big)^r\Vert P_{N_1}\tilde{f}\Vert_{X^{s,r,\frac{1}{2}+}}\Vert P_{M_2}\tilde{g}\Vert_{X^{s,r,\frac{1}{2}+}}\Big]^2\Big)^{\frac{1}{2}}=:T^{\frac{1}{4}}I.    \end{aligned}
\end{equation*}
as long as $s,r>\frac{d}{2}-\frac{1}{4}$. In this case, $I$ could be treated as in Case A.$1$ by switching the roles of $M_1$ and $M_2$, and letting $\beta_{N_1,M_2}=\Vert P_{N_1}\tilde{f}\Vert_{X^{s,r,\frac{1}{2}+}}\Vert P_{M_2}\tilde{g}\Vert_{X^{s,r,\frac{1}{2}+}}$.
\\{\bf Case C.2.} $M_1\geq N_2^{d-1}\log N_2$.\\
By the Lemma \ref{L_xiestimate}, Hölder, Bernstein, and Strichartz inequalities, we get
\begin{equation*}
  \begin{aligned}
 \Vert P_{N,M}\tilde{Q}^+&(P_{a+(-N_2,N_2]^d}P_{N_1,M_1}\tilde{f},P_{N_2,M_2}\tilde{g})\Vert_{L^2_{T}L^2_{x\xi}}\\&\lesssim T^{\frac{1}{4}}N_2^{\frac{d}{4}}M_1^{\frac{d}{4}}\Vert P_{a+(-N_2,N_2]^d}P_{N_1,M_1}\tilde{f}\Vert_{L^{4}_{T}L^{4}_{x\xi}}\Vert P_{N_2,M_2}\tilde{g}\Vert_{X^{0,\frac{1}{2}+}}\\&\lesssim T^{\frac{1}{4}}N_2^{\frac{d}{4}}M_1^{\frac{d}{2}}\Vert P_{a+(-N_2,N_2]^d}P_{N_1,M_1}\tilde{f}\Vert_{X^{0,\frac{1}{2}+}}\Vert P_{N_2,M_2}\tilde{g}\Vert_{X^{0,\frac{1}{2}+}}  
  \end{aligned}  
\end{equation*}
which implies for $s>\frac{d}{4},\, r>\frac{d}{2}$:
\begin{equation*}
    \begin{aligned}
 &\text{RHS of}\,\, \eqref{almostorthgain}\\&\lesssim T^{\frac{1}{4}}\Big(\sum_{N,M}N^{2s}M^{2r}\Big[\sum_{\substack{N_1\gtrsim N, M_2\gtrsim M\\N_2, M_1}}N_2^{\frac{d}{4}}M_1^{\frac{d}{2}}\Vert P_{N_2,M_2}\tilde{g}\Vert_{X^{0,\frac{1}{2}+}}\Big(\sum_{a\in\mathbb{Z}^d}\Vert P_{a+(-N_2,N_2]^d}P_{N_1,M_1}\tilde{f}\Vert_{X^{0,\frac{1}{2}+}}^2\Big)^{\frac{1}{2}}\Big]^2\Big)^{\frac{1}{2}}\\&\sim T^{\frac{1}{4}}\Big(\sum_{N,M}\Big[\sum_{\substack{N_1\gtrsim N, M_2\gtrsim M\\N_2, M_1}}\Big(\frac{N}{N_1}\Big)^s\Big(\frac{M}{M_2}\Big)^rN_2^{\frac{d}{4}-s}M_1^{\frac{d}{2}-r}\Vert P_{N_1,M_1}\tilde{f}\Vert_{X^{s,r,\frac{1}{2}+}}\Vert P_{N_2,M_2}\tilde{g}\Vert_{X^{s,r,\frac{1}{2}+}}\Big]^2\Big)^{\frac{1}{2}}\\&\lesssim T^{\frac{1}{4}}\Big(\sum_{N,M}\Big[\sum_{\substack{N_1\gtrsim N\\ M_2\gtrsim M}}\Big(\frac{N}{N_1}\Big)^s\Big(\frac{M}{M_2}\Big)^r\Vert P_{N_1}\tilde{f}\Vert_{X^{s,r,\frac{1}{2}+}}\Vert P_{M_2}\tilde{g}\Vert_{X^{s,r,\frac{1}{2}+}}\Big]^2\Big)^{\frac{1}{2}}\\&\lesssim T^{\frac{1}{4}}\Vert \tilde{f}\Vert_{X^{s,r,\frac{1}{2}+}}\Vert \tilde{g}\Vert_{X^{s,r,\frac{1}{2}+}}.    \end{aligned}
\end{equation*}
\\{\bf Case D.} $N_2\geq N_1$, $M_1\geq M_2$.\\
This case can be handled as that in the Case C. by swapping the roles of $\tilde{f}$ and $\tilde{g}$. So, the proof is completed.
\end{proof}
\subsection{Proof of Theorem \ref{maintheorem}}
Firstly, let us recall the following standard Fourier restriction norm estimates.
\begin{lemma}\label{basicX^s,bestimates}
  Let $s,r\in \mathbb{R}$ and $\frac{1}{2}<b\leq 1$. Then, we 
  have 
\begin{align}
      \Vert \psi(t)S(t)g\Vert_{X^{s,r,b}}&\lesssim \Vert g\Vert_{H^s_xH^r_{\xi}}, \label{b2} \\ \Vert \psi(t)\int_0^tS(t-t')F(g)dt'\Vert_{X^{s,r,b}}&\lesssim \Vert F(g)\Vert_{X^{s,r,b-1}}, \label{b3}\\ \Vert g\Vert_{X^{s,r,b-1}}&\lesssim \Vert g\Vert_{L_t^pH^s_xH^r_{\xi}}, \quad p\in(1,2],\,\, b\leq \frac{3}{2}-\frac{1}{p}.\label{b4}
  \end{align}    
\end{lemma}
We only show the boundedness of the contraction map on the ball, which is defined below, for the complete treatment of the conditions of well-posedness for the Boltzmann equation, see \cite{CDP,CDPmoment,Chen-Holmer,Chen-Shen-Zhang}. The contraction argument is applied to the operator
\begin{align*}
    \Gamma \tilde{f}=\psi(t)S(t)\tilde{f}_0+\psi(t)\int_0^tS(t-t')\psi_T(t')\tilde{Q}(\tilde{f}(t'),\tilde{f}(t'))dt'
\end{align*}
on the ball 
\begin{align*}
   B=\{\tilde{f}\in X^{s,r,b}:\Vert \tilde{f}\Vert_{X^{s,r,b}}\leq R\} 
\end{align*}
 where $R=2C\Vert \tilde{f}_0\Vert_{H^{s}_xH^r_{\xi}}$. Note that by the estimate \eqref{b4}, Hölder inequality, Lemmas \ref{Q^-} and \ref{Q^+}, we have
\begin{equation}\label{nonlinearityest}
 \begin{aligned}
    \Vert\psi_T(t)\tilde{Q}(\tilde{f},\tilde{f}) \Vert_{X^{s,r,b-1}}&\lesssim \Vert \psi_T(t)\tilde{Q}(\tilde{f},\tilde{f})\Vert_{L^{\frac{2}{3-2b}}_tH^s_xH^r_{\xi}}\\&\lesssim T^{1-b}\big(\Vert \psi_T(t)\tilde{Q}^-(\tilde{f},\tilde{f})\Vert_{L^2_tH^s_xH^r_{\xi}} + \Vert \psi_T(t)\tilde{Q}^+(\tilde{f},\tilde{f})\Vert_{L^2_tH^s_xH^r_{\xi}}\big)\\&\lesssim T^{\frac{5}{4}-b}\Vert \tilde{f}\Vert_{X^{s,r,b}}^2.
\end{aligned}   
\end{equation}
Then, from \eqref{b2}, \eqref{b3}, and \eqref{nonlinearityest} we obtain, for $\tilde{f}\in B$, that
\begin{align*}
    \Vert \Gamma \tilde{f}\Vert_{X^{s,r,b}}&\leq \Vert \psi(t)S(t)\tilde{f}_0\Vert_{X^{s,r,b}}+\Vert \psi(t)\int_0^tS(t-t')\psi_T(t')\tilde{Q}(\tilde{f}(t'),\tilde{f}(t'))dt'\Vert_{X^{s,r,b}}\\&\leq C\Vert \tilde{f}_0\Vert_{H^s_xH^r_{\xi}}+C\Vert\psi_T(t)\tilde{Q}(\tilde{f},\tilde{f}) \Vert_{X^{s,r,b-1}}\\&\leq \frac{R}{2}+CT^{\frac{5}{4}-b}\Vert \tilde{f}\Vert_{X^{s,r,b}}^2\\&\leq R 
\end{align*}
by taking $T>0$ sufficiently small so that $CT^{\frac{5}{4}-b}R\leq \frac{1}{2}$.
\section*{Acknowledgement}
E.B. and Y.W. were supported by the EPSRC New Investigator Award (grant no. EP/V003178/1).
C.S and N.T. were partially
supported by the ANR project Smooth ANR-22-CE40-0017.

\nocite{*}
\bibliographystyle{abbrv}
\bibliography{reference.bib}

\end{document}